\documentclass[11pt,reqno]{amsart}    

\usepackage[margin=37mm]{geometry}

\usepackage{amssymb,amsmath,amsfonts}
\usepackage{amsthm, amscd, mathrsfs, helvet, mathtools}
\usepackage{bm, stmaryrd}
\usepackage{setspace}
\usepackage{graphicx,verbatim}
\usepackage[usenames, dvipsnames]{color}
\usepackage{color}
\usepackage[mathcal]{euscript}
\usepackage{frcursive}
\usepackage[T1]{fontenc}

\usepackage{enumerate}
\usepackage{tikz}    
\usetikzlibrary{arrows,matrix,snakes}
\usetikzlibrary{decorations.markings,shapes.geometric,shapes.misc}    

\theoremstyle{plain}
\newtheorem{theorem}{Theorem}[section]
\newtheorem*{theorem*}{Theorem}
\newtheorem{proposition}[theorem]{Proposition}
\newtheorem*{proposition*}{Proposition}

\newtheorem{lemma}[theorem]{Lemma}

\theoremstyle{remark}

\newtheorem{remarkx}{Remark}
\newtheorem*{remarkxstar}{Remark}

\newenvironment{remark*}
  {\pushQED{\qed}\remarkxstar}
  {\popQED\endremarkx}

\definecolor{brightRed}{rgb}{1,0,0}

\definecolor{brightBlue}{rgb}{0,0,1}

\definecolor{myGreen}{rgb}{0,0.7,0}

\newcommand{\tensor}[1]{{\mathfrak{#1}}}

\newcommand{\fm}[1]{#1}

\DeclareMathOperator{\tr}{tr} 
\DeclareMathOperator{\res}{res}

\DeclareMathOperator{\End}{End}

\DeclareMathOperator{\Hom}{Hom}
\DeclareMathOperator{\hox}{\hat\ox}
\DeclareMathOperator{\rank}{rank}

\DeclareSymbolFont{widetriangleaccent}{OMX}{yhex}{m}{n}
\DeclareMathAccent{\widetriangle}{\mathord}{widetriangleaccent}{"E6}

\newcommand{\longhookrightarrow}{\lhook\joinrel\relbar\joinrel\rightarrow}

\def\cent{\mathsf k}
\def\cocent{\mathsf d}

\def\a{\mathfrak{a}}
\def\b{\mathfrak{b}}
\def\c{\mathfrak{c}}

\def\g{\mathfrak{g}}

\def\h{\mathfrak{h}}

\def\n{\mathfrak{n}}

\def\s{\varsigma}

\def\VV{\mathbb{V}}
\def\MM{\mathbb{M}}

\def\Q{q}

\def\ZZ{\mathbb{Z}}

\def\ha{\mbox{\small $\frac{1}{2}$}}
\def\ta{\mbox{\small $\frac{1}{3}$}}
\def\qa{\mbox{\small $\frac{1}{4}$}}
\def\sa{\mbox{\small $\frac{1}{6}$}}

\newcommand{\bb}[1]{[\kern-.1em[ #1 ]\kern-.1em]}
\newcommand{\lau}[1]{(\kern-.2em( #1 )\kern-.2em)}

\newcommand{\lbf}[1]{\langle\kern-.2em\langle #1 \rangle\kern-.2em\rangle}
\newcommand{\biglbf}[1]{\big\langle \kern-.25em \big\langle #1 \big\rangle \kern-.25em \big\rangle}

\newcommand{\rbf}[1]{(\kern-.2em( #1 )\kern-.2em)}
\newcommand{\bigrbf}[1]{\big( \kern-.25em \big( #1 \big) \kern-.25em \big)}

\def\CC{\mathbb{C}}

\def\P{\mathcal{P}}

\def\M{\mathcal{M}}

\def\sl{\mathfrak{sl}}

\newcommand{\vac}{\left|0\right>}
\newcommand{\vack}{\left|0\right>^{\!k}}
\newcommand{\vackk}{\left|0\right>^{\!\bm k}}

\newcommand{\slMhat}{\widehat{\mathfrak{sl}}_M}
\newcommand{\slM}{\mathfrak{sl}_M}

\newcommand{\be}{\begin{equation}}
\newcommand{\ee}{\end{equation}}
\newcommand{\nn}{\nonumber}
\newcommand{\hc}{M}
\newcommand{\del}{\partial}
\newcommand{\ox}{\otimes}
\newcommand{\nord}[1]{:\! #1 \!:}
\newcommand{\longto}{\longrightarrow}
\newcommand{\8}\infty

\newcommand{\into}{\hookrightarrow}
\newcommand{\longinto}{\longhookrightarrow}

\newcommand{\lsln}{{{}^L\slMhat}}
\newcommand{\homog}{{\textup{hom}}}
\newcommand{\princ}{{\textup{prin}}}
\newcommand{\vacl}{\left|\bm\lambda\right>}

\newcommand{\DD}[2]{D_{#1}^{(#2)}}
\newcommand{\mc}{\mathcal}
\newcommand{\aaa}[1]{\tfrac{2}{\hc} k_{#1}}
\newcommand{\ff}[2]{(#1)_{#2}}
\newcommand{\T}{\omega}
\newcommand{\W}{\mathsf W}

\def\lg{{{}^L\!\g}}
\def\ln{{}^L\n}

\def\lb{{}^L\b}
\def\lh{{}^L\h}

\def\1{\tensor{1}}
\def\2{\tensor{2}}
\def\3{\tensor{3}}
\def\4{\tensor{4}}

\numberwithin{equation}{section}

\author{Sylvain Lacroix}
\address{Univ Lyon, Ens de Lyon, Univ Claude Bernard, CNRS, Laboratoire de Physique, F-69342 Lyon, France}
\email{sylvain.lacroix@ens-lyon.fr}

\author{Beno\^{\i}t Vicedo}
\address{Department of Mathematics, University of York, York YO10 5DD, U.K.}
\email{benoit.vicedo@gmail.com}

\author{Charles Young}
\address{School of Physics, Astronomy and Mathematics, University of Hertfordshire, College Lane, Hatfield AL10 9AB, UK.}
\email{c.a.s.young@gmail.com}

\begin{document} 
\title[Cubic hypergeometric integrals of motion in affine Gaudin models]{Cubic hypergeometric integrals\\ of motion in affine Gaudin models}

\begin{abstract}
We construct cubic Hamiltonians for quantum Gaudin models of affine types $\slMhat$. They are given by hypergeometric integrals of a form we recently conjectured in \cite{LVY1}. We prove that they commute amongst themselves and with the quadratic Hamiltonians. We prove that their vacuum eigenvalues, and their eigenvalues for one Bethe root, are given by certain hypergeometric functions on a space of affine opers. 
\end{abstract}

\maketitle

\section{Introduction}
The \emph{quantum Gaudin model} \cite{GaudinBook} can be defined for any symmetrizable Kac-Moody Lie algebra $\g$. One chooses a collection $z_1,\dots, z_N$ of distinct points in the complex plane and the model is defined by its quadratic Hamiltonians, which are the elements
\be \mc H_i = \sum_{\substack{j=1\\j\neq i}}^N \frac{\Xi^{(ij)}}{z_i-z_j},\qquad i=1,\dots,N, \nn\ee
of the (suitably completed) tensor product $U(\g^{\oplus N})$, where $\Xi$ is the canonical element of $\g\ox\g$ coming from the invariant bilinear form on $\g$. In finite types $\g$, these quadratic Hamiltonians are known to belong to a large commutative subalgebra $\mc B \subset U(\g^{\oplus N})$ called the \emph{Bethe} or \emph{Gaudin}  algebra, which is generated by the $\mc H_i$ together with the central elements of $U(\g^{\oplus N})$ and also (when $\rank(\g)>1$) certain families of \emph{higher Gaudin Hamiltonians} \cite{FFR,Fopers,Talalaev, MTV1, MR2424092, MR3004777,RybnikovProof}.

When $\g$ is of affine type, a general construction of the non-local higher Hamiltonians in such affine Gaudin models was given in \cite{FFsolitons}. Moreover, of particular relevance to the present work is the conjecture made in \cite{FFsolitons} that, for $N=2$ and $\g = \widehat{\mathfrak{sl}}_2$, the local higher affine Gaudin Hamiltonians are given by the integrals of motion of the coset Virasoro algebra. The case when $\g$ is of untwisted affine type is of particular interest since it is expected that quantum integrable field theories can be described as affine Gaudin models associated with such Kac-Moody Lie algebras \cite{FFsolitons, V17}. In a recent paper, \cite{LVY1}, we gave a broad conjecture for the form the local higher Hamiltonians should take for $\g$ of untwisted affine type, as well as a precise conjecture for their eigenvalues on tensor products of irreducible highest-weight $\g$-modules. According to \cite{LVY1} both the local higher Hamiltonians and their eigenvalues should be given by integrals of hypergeometric type in the spectral plane (\emph{i.e.} the copy of $\CC$ containing the marked points $z_1,\dots,z_N$). 

In the present paper, we check the simplest case of those conjectures. The prediction is that there is a family of local higher Hamiltonians for each \emph{exponent} of $\g$. Recall that, when $\g$ is of affine type, the exponents are a countably infinite (multi)set of integers whose pattern repeats modulo the Coxeter number (see \emph{e.g.} \cite{KacBook}). The quadratic Hamiltonians are associated with the exponent 1. The next case to check is that of Hamiltonians associated to cubic symmetric invariant tensors on the underlying finite-type Lie algebra. Such tensors in fact exist only in types $\mathsf A_{M-1}$ (\emph{i.e.} $\sl_{M}$) with $M\geq 3$. Equivalently, ${}^1\!\mathsf A_{M-1}$ are the only untwisted affine types for which 2 is an exponent. In this paper we specialize to those types. We construct the cubic Hamiltonians and show that they commute amongst themselves and with the quadratic Hamiltonians. We also show that their eigenvalues are as predicted in \cite{LVY1} at least for Bethe vectors corresponding to 0 or 1 Bethe roots. 

\bigskip

The paper is structured as follows. 

In \S\ref{sec: slMhat} we recall details of the affine algebra $\slMhat$, its local completion and its vacuum Verma module. We also recall very briefly some concepts about vertex algebras. 

In \S\ref{sec: Hams} we define the algebra of observables of the quantum Gaudin model, and define the states $\s_1(z)$ and $\s_2(z)$ used to construct the quadratic and cubic Hamiltonians respectively. The main technical results of the paper are then Theorems \ref{thm: S1S2 zeroth prod} and \ref{thm: gsym}. The former shows that the ${}_{(0)}$th products (in the sense of vertex algebras) between these states vanish modulo certain \emph{twisted derivatives} and \emph{translates}. The latter shows that the same is true of the action of the diagonal copy of $\slMhat$ on $\s_i(z)$, $i=1,2$. 

These statements allow us to prove our main result, Theorem \ref{thm: main}, which establishes that the 
quadratic and cubic Hamiltonians commute amongst themselves and with the diagonal action of $\slMhat$. The Hamiltonians  $\hat Q_i^\gamma$, $i=1,2$, are defined in \eqref{def: Hig}: the superscript $\gamma$ denotes a Pochhammer contour in $\CC\setminus\{z_1,\dots,z_N\}$. 

(There is a slight subtlety because the $\hat Q_1^\gamma$ are not exactly the standard quadratic Hamiltonians $\mc H_j$. Thus, we also show in Theorem \ref{thm: main} that the $\hat Q_i^\gamma$, $i=1,2$, commute with the $\mc H_j$. We do so using another result, Theorem \ref{thm: S1 0 cal Si}.) 

In \S\ref{sec: op and eval} we recall the conjectured eigenvalues for the Hamiltonians   $\hat Q_i^\gamma$, $i=1,2$, from \cite{LVY1}. Namely, the eigenvalues are obtained by putting a certain \emph{affine oper} (coming from an affine \emph{Miura oper}) into \emph{quasi-canonical form}, and then integrating the resulting coefficient functions $v_i(z)$, $i=1,2$, along the same contour $\gamma$. 

In \S\ref{sec: BA} we check that the predicted eigenvalues of the cubic Hamiltonians are correct (for the quadratic Hamiltonians see \cite{LVY1}) for Bethe vectors with $0$ and $1$ Bethe roots, \emph{i.e.} for the vacuum state and for Bethe eigenstates at one step down in the principal gradation. 

Finally, in \S\ref{sec: two-point} we consider the special case of only $N=2$ marked points. In that case we show that our Hamiltonians $\hat Q_i^\gamma$, $i=1,2$ coincide up to rescaling with the zero modes of, respectively, the GKO coset conformal vector $\omega$, \cite{GKO}, and a state $\W$ constructed in \cite{BBSS}. It is known that, after certain further specializations, these generate a copy of the $W_3$ algebra. This provides an interesting arena for checking conjectures about higher Hamiltonians, since higher integrals of motion of the $W_3$ algebra are already known. Our results in this two point case generalise to $\widehat{\mathfrak{sl}}_3$ the corresponding statements made in \cite{FFsolitons} for the local quadratic Hamiltonian of the affine Gaudin model associated with $\widehat{\mathfrak{sl}}_2$.

\subsubsection*{Acknowledgements}
CY is grateful to E. Mukhin for interesting discussions. 
SL thanks F. Delduc and M. Magro for interesting discussions. 
This work is partially supported by the French Agence Nationale de la Recherche (ANR) under grant ANR-15-CE31-0006 DefIS.

\section{Vacuum verma modules for $\slMhat$}\label{sec: slMhat}
\subsection{Loop realization of $\slMhat$} \label{sec: loop rel slM}

We work over $\CC$. Pick and fix an integer $M\geq 3$. 
Let $\slM[t,t^{-1}] = \slM\ox\CC[t,t^{-1}]$ denote the Lie algebra of Laurent polynomials, in a formal variable $t$, with coefficients in the finite-dimensional simple Lie algebra $\slM$. 
The Lie bracket on $\slM[t,t^{-1}]$ is given by $[a \ox f(t), b \ox g(t) ] \coloneqq [a,b] \ox f(t) g(t)$ for any $a, b \in \slM$ and $f(t), g(t) \in \CC[t, t^{-1}]$. Let $(\cdot|\cdot): \slM\times\slM \to \CC$ be the standard bilinear form on $\slM$. It is given by 
\be (X|Y) \coloneqq \tr(XY)\label{def: bf}\ee
where $\tr$ denotes the trace in the defining $M\times M$ matrix representation.
The affine algebra $\slMhat$ is the central extension of $\slM[t,t^{-1}]$ by a one dimensional centre $\CC\cent$,
\be 0 \longto \CC \cent \longto \slMhat \longto \slM[t,t^{-1}] \longto 0, \nn\ee
whose commutation relations are given by $[\cent,\cdot] = 0$ and 
\be [a \ox f(t), b \ox g(t) ] \coloneqq[a,b] \ox f(t) g(t) - (\res_tfdg) (a|b) \cent .\nn\ee
Let $a_n \coloneqq a\ox t^n$ for $a\in \slM$ and $n\in \ZZ$. The commutation relations can equivalently be written as
\be [a_m, b_n ] = [a,b]_{n+m} - n \delta_{n+m,0} (a|b) \cent .\nn\ee
Define the Lie algebra 
\be \g \coloneqq \slMhat \oplus \CC\cocent,\nn\ee
by declaring that the \emph{derivation element} $\cocent$ obeys $[\cocent,\cent]=0$ and $[\cocent, a\otimes f(t)] = a\otimes t\del_t f(t)$ for all $a \in \slM$ and $f(t) \in \CC[t,t^{-1}]$.
\subsection{Kac-Moody data}\label{sec: uaa}
Recall, for example from \cite{KacBook}, that the Lie algebra $\g$ is isomorphic to the Kac-Moody algebra over $\CC$ of type ${}^1\!\mathsf A_{M-1}$. The Cartan matrix is $A\coloneqq (a_{ij})_{i,j=0}^{M-1} = ( 2\delta_{ij} - \delta_{i+1,j} - \delta_{i-1,j} )_{i,j=0}^{M-1}$, where addition of indices is modulo $M$. Fix a Cartan decomposition $\g = \n_- \oplus \h \oplus \n_+$ and sets of Chevalley-Serre generators $\{e_i\}_{i=0}^{M-1}\subset \n_+$, $\{f_i\}_{i=0}^{M-1}\subset \n_-$. 
Let $\{\check\alpha_i\}_{i=0}^{M-1}\subset \h$ be the simple coroots of $\g$ and $\{\alpha_i\}_{i=0}^{M-1}\subset \h^*$ the simple roots. They obey $a_{ij} = \langle \alpha_j, \check \alpha_i\rangle$ for $i,j\in \{0,1,\dots,{M-1}\}$, where $\langle \cdot,\cdot\rangle : \h^* \times \h \to \CC$ is the canonical pairing of the Cartan subalgebra $\h$ and its dual $\h^*$. 
The defining relations of $\g$ are then
\begin{subequations} \label{KM relations}
\begin{alignat}{2}
\label{KM rel a} [x, e_i] &= \langle \alpha_i,x\rangle  e_i, &\qquad
[x,  f_i] &= - \langle \alpha_i,x\rangle f_i, \\
\label{KM rel b} [x, x'] &= 0, &\qquad
[e_i, f_j] &= \check\alpha_i \delta_{i,j},
\end{alignat}
for any $x, x' \in \h$ and $i,j \in \{ 0, 1, \ldots, M-1 \}$, together with the Serre relations
\begin{equation} \label{KM rel c}
(\text{ad}\, e_i)^{1- a_{ij}} e_j = 0, \qquad (\text{ad}\, f_i)^{1- a_{ij}} f_j = 0.
\end{equation}
\end{subequations}
The centre of $\g$ is one dimensional and spanned by the \emph{central element} $\cent \coloneqq \sum_{i=0}^{M-1}\check \alpha_i$.
The Cartan subalgebra has a basis consisting of the simple coroots $\{\check\alpha_i\}_{i=0}^{M-1}$ together with  the derivation element $\cocent$, which obeys
\be \langle \alpha_i, \cocent \rangle = \delta_{i,0}. \nn\ee
(This condition fixes $\cocent$ uniquely up to the addition of a multiple of $\cent$.) 

The matrix $(a_{ij})_{i,j=1}^{M-1}$ obtained by removing the zeroth row and column of $A$ is the Cartan matrix of finite type $\mathsf A_{M-1}$, with $\slM$ the corresponding finite-dimensional simple Lie algebra. We can identify $\slM$ with the subalgebra of $\g$ generated by $\{e_i\}_{i=1}^{M-1}\subset \n_+$ and $\{f_i\}_{i=1}^{M-1}\subset \n_-$, and then the defining relations are as above.
Let $\dot\h \coloneqq \textup{span}_\CC \{\check\alpha_i\}_{i=1}^{M-1}\subset \slM$ denote its Cartan subalgebra and $\dot\h^* = \textup{span}_\CC \{\alpha_i\}_{i=1}^{M-1}$ its dual.
%
%
Let $\{\omega_i\}_{i=1}^{M-1} \subset \dot\h^*$ and $\{\check\omega_i\}_{i=1}^{M-1} \subset \dot \h$ be respectively the fundamental weights and coweights of $\slM$.

\subsection{Local completion and vacuum Verma module}
For any $k\in \CC$, let $U_k(\slMhat)$ denote quotient of the enveloping algebra of $\slMhat$ by the two-sided ideal generated by $\cent - k$. For each $n\in \ZZ_{\geq 0}$ define the left ideal $J_n\coloneqq U_k(\slMhat) \cdot (\slM \ox t^n\CC[t])$. 
That is, $J_n$ is the linear span of monomials of the form $a_p\dots b_q c_r$ with $r\geq n$, $c \in \slM$, and some non-negative number of elements $a,\dots,b \in\slM$ and mode numbers $p,\dots,q\in \ZZ$.
The inverse limit $\tilde U_k(\slMhat) \coloneqq \varprojlim U_k(\slMhat) \big/ J_n$ 
is a complete topological algebra, called the \emph{local completion of $U(\slMhat)$ at level $k$}. By definition, elements of $\tilde U_k(\slMhat)$ are (possibly infinite) sums $\sum_{m\geq 0} X_m$ of elements $X_m\in U_k(\slMhat)$ which truncate to finite sums when one works modulo any $J_n$, \emph{i.e.} for every $n$, $X_m\in J_n$ for all sufficiently large $m$.

A module $\M$ over $\slMhat$ is called \emph{smooth} if, for all $a \in \slM$ and all $v\in \M$, $a_nv=0$ for all sufficiently large $n$. A module $\M$ has \emph{level $k$} if $\cent-k$ acts as zero on $\M$. Any smooth module of level $k$ over $\slMhat$ is also a module over $\tilde U_k(\slMhat)$.  

We have the Lie subalgebra $\slM[t] \oplus \CC \cent \subset \slMhat$. Let $\CC \vack$ denote the one-dimensional representation of this Lie algebra given by $(\cent-k) \vack=0$ and $a_n \vack = 0$ for all $n\geq 0$ and all $a\in \slM$. Let $\VV_0^k$ denote the induced $\slMhat$-module:
\be \VV_0^k = U(\slMhat) \ox_{U(\slM[t] \oplus \CC \cent)} \CC \vack. \nn\ee
This module $\VV_0^k$ is called the \emph{vacuum Verma module of level $k$}. 
It is a smooth module. Concretely, $\VV_0^k$ is the linear span of vectors of the form $a_p \dots b_q \vack$ with $a,\dots,b\in \slM$ and strictly negative mode numbers $p,\dots,q\in \ZZ_{\leq -1}$. Elements of $\VV_0^k$ are called \emph{states}.

Let $[T,\cdot]$ be the derivation on $U_k(\slMhat)$ defined by $[T, a_n] \coloneqq -n a_{n-1}$ and $[T, 1] \coloneqq 0$. By setting $T(X\vack) \coloneqq [T, X] \vack$ for any $X \in U_k(\slMhat)$, one can then regard $T$ also as a linear map  $\VV_0^k\to \VV_0^k$, called the \emph{translation operator}.

\subsection{Vertex algebra structure}

For every state $A\in \VV_0^k$ and every $n\in \ZZ$, there is an element $A_{(n)}\in \tilde U_k(\slMhat)$, the \emph{$n^{\rm th}$ formal mode of $A$}. These modes can be arranged into a formal power series
\be Y[A, u] \coloneqq \sum_{n\in \ZZ} A_{(n)} u^{-n-1} \label{def: Yformal}\ee
where $Y[\cdot, u]$ is called the \emph{formal state-field map}. Their definition is as follows.

First, if $A = a_{-1} \vack$ for some $a\in \slM$ then $A_{(n)} \coloneqq a_n$ for all $n\in \ZZ$, \emph{i.e.}
\be Y[a_{-1} \vack, u] \coloneqq \sum_{n\in \ZZ} a_n u^{-n-1}.\nn\ee 
Next, for all states $A,B\in \VV_0^k$,
\be Y[TA, u] \coloneqq \del_u Y[A, u]\qquad\text{and}\qquad  Y[A_{(-1)} B, u] \coloneqq\,\,\, \nord{Y[A, u] Y[B, u]} \nn\ee
where 
\be \nord{Y[A,u] Y[B,u]} \,\,\, \coloneqq \Bigg(\sum_{m<0} A_{(m)} u^{-m-1} \Bigg)Y[B,u] + Y[B,u] \Bigg(\sum_{m\geq 0} A_{(m)} u^{-m-1}\Bigg)\label{nord}\ee 
is called the \emph{normal ordered product}. These assignments together recursively define $Y[C, u]$ for all $C\in \VV_0^k$ (by writing $C = a_{-n} B$ for some $a\in \slM$, $n\in \ZZ_{\geq 1}$ and $B\in \VV_0^k$).  

The \emph{state-field map} $Y(\cdot,u) : \VV_0^k \to \Hom(\VV_0^k, \VV_0^k((u)))$ is obtained from $Y[A, u]$ by sending each formal mode to its image in $\End(\VV_0^k)$. This map $Y(\cdot, u)$ obeys a collection of axioms that make $\VV_0^k$ into a \emph{vertex algebra}; see \emph{e.g.} \cite{FrenkelBenZvi}. 

\section{Quadratic and cubic Hamiltonians}\label{sec: Hams}
  
\subsection{The algebra of observables} \label{sec: alg obs}
Let $\bm k \coloneqq (k_i)_{i=1}^N$ be a collection of $N \in \ZZ_{\geq 1}$ complex numbers $k_i \neq - \hc$ for $i = 1, \ldots, N$. Consider the tensor product
\be \VV_{0}^{\bm k} \coloneqq \VV_0^{k_1} \ox \dots \ox \VV_0^{k_N} \nn\ee
of vacuum Verma modules. We can regard it as a module over the direct sum $\slMhat^{\oplus N}$ of $N$ copies of $\slMhat$. 
Let $A^{(i)}\in \slMhat^{\oplus N}$ denote the copy of $A\in \slMhat$ in the $i^{\rm th}$ direct summand.
Let $\CC \vackk$ denote the one-dimensional representation of the Lie subalgebra $(\slM[t] \oplus \CC \cent)^{\oplus N} \subset \slMhat^{\oplus N}$ defined by  $(\cent^{(i)}-k_i) \vackk=0$ and $a^{(i)}_n \vack = 0$, for all $n\geq 0$, all $a \in \slM$, and all $i\in \{1,\dots,N\}$. 
Then $\VV_0^{\bm k}$ is the induced $\slMhat^{\oplus N}$-module:
\be \VV_0^{\bm k} = U(\slMhat^{\oplus N}) \ox_{U( (\slM[t] \oplus \CC \cent)^{\oplus N})} \CC \vackk. \nn\ee
Let $U_{\bm k}(\slMhat^{\oplus N})$ denote the quotient of $U(\slMhat^{\oplus N})$ by the two-sided ideal generated by $\cent^{(i)} - k_i$ for all $i\in\{1,\dots,N\}$. We have the isomorphism 
\be U_{\bm k}(\slMhat^{\oplus N}) \cong  U_{k_1}(\slMhat) \ox \dotsm \ox  U_{k_N}(\slMhat),\nn\ee
which sends $A^{(i)} \in \slMhat^{\oplus N} \subset  U_{\bm k}(\slMhat^{\oplus N})$ to $1^{\ox i-1} \ox A \ox 1^{\ox N-i}$. 
Let $J^N_n$ denote the left ideal in $U_{\bm k}(\slMhat^{\oplus N})$ generated by $a_r^{(i)}$ for all $r \geq n$, $a\in \slM$ and all $i\in\{1,\dots,N\}$. Let $\tilde U_{\bm k}(\slMhat^{\oplus N}) \coloneqq \varprojlim U_{\bm k}(\slMhat^{\oplus N})\big/J^N_n$ denote the inverse limit. It is a complete topological algebra, and 
\be \tilde U_{\bm k}(\slMhat^{\oplus N}) \cong 
\tilde U_{k_1}(\slMhat) \hox \dotsm \hox \tilde U_{k_N}(\slMhat),\nn\ee
where $\hox$ denotes the completed tensor product. We call $\tilde U_{\bm k}(\slMhat^{\oplus N})$ the \emph{algebra of observables} of the Gaudin model. 

We have the formal state-field map $Y[\cdot, u]:  \VV_0^{\bm k} \to \tilde U_{\bm k}(\slMhat^{\oplus N})[[u, u^{-1}]]$ and translation operator $T\in \End(\VV_0^{\bm k})$ defined as above. $\VV_{0}^{\bm k}$ is a module over $\tilde U_{\bm k}(\slMhat^{\oplus N})$, and we get the state-field map $Y(\cdot,u) : \VV_0^{\bm k} \to \Hom(\VV_0^{\bm k}, \VV_0^{\bm k}((u)))$.

Let $\Delta : \slMhat \into \slMhat^{\oplus N}$ denote the diagonal embedding of $\slMhat$ into $\slMhat^{\oplus N}$,
\be \Delta x = \sum_{i=1}^N x^{(i)} \nn\ee 
It extends uniquely to an embedding $\Delta: U(\slMhat) \into U(\slMhat^{\oplus N})= U(\slMhat)^{\ox N}$ (the usual $N$-fold coproduct). For all $a, b \in \slM$ and $m,n\in \ZZ$, 
\be [\Delta a_m, \Delta b_n] = \Delta[a,b]_{n+m} - n \delta_{n+m,0} (a | b) \sum_{i=1}^n \cent^{(i)}\nn\ee
in $U(\slMhat^{\oplus N})$. 
Therefore $\Delta$ descends to an embedding of the quotients, $\Delta : U_{|\bm k|}(\slMhat) \into U_{\bm k}(\slMhat^{\oplus N})$, where $|\bm k| = \sum_{i=1}^N k_i$, and thence of their completions,
\be \Delta : \tilde U_{|\bm k|}(\slMhat) \longinto \tilde U_{\bm k}(\slMhat^{\oplus N}). \nn\ee

\subsection{Invariant tensors on $\slM$}
Let $I^a$, $a=1,\dots,\dim(\slM)$, be a basis of $\slM$ and $I_a$, $a=1,\dots,\dim(\slM)$ the dual basis with respect to the non-degenerate bilinear form $(\cdot|\cdot)$ from \eqref{def: bf}. 
Let $f^{ab}{}_c$ denote the structure constants, 
\be [I^a,I^b] = f^{ab}{}_c I^c \coloneqq \sum_{c} f^{ab}{}_c I^c,\ee 
where, from now on, we employ summation convention over repeated pairs of Lie algebra indices.
Since the symmetric bilinear form $(\cdot|\cdot)$ is non-degenerate, we may and shall 
suppose that the basis is chosen such that
\be (I_a|I_b) = \delta_{ab}. \nn\ee
Then $I_a=I^a$ for each $a$ and we no longer need to distinguish between upper and lower indices, and we shall write for example $f_{abc} = f^{ab}{}_c$. 
We have  
\be \delta_{ab} = \tr I_a I_b, \qquad f_{abc} = \tr [I_a,I_b] I_c = \tr(I_a I_b I_c- I_bI_aI_c). \nn\ee

Recall that in general a tensor $t: \slM \times \dots \times \slM \to \CC$ is \emph{invariant} if 
\be t([a,x],y,\dots,z) + t(x,[a,y],\dots,z) + \dots + t(x,y,\dots,[a,z])=0 \nn\ee
for all $x,y,\dots,z$ and $a$ in $\slM$. In particular, $\delta_{ab}$ and $f_{abc}$ are the components of respectively a symmetric second-rank invariant tensor and a totally skew-symmetric third-rank invariant tensor.

The reason for specialising to $\sl_M$ in this paper is that in these types, and in no others, there exists a nonzero totally symmetric third-rank tensor. It is unique up to normalization, and is given by $t(x,y,z) \coloneqq \tr(xyz + yxz)$. Let $t_{abc}$ denote its components,
\be t_{abc} \coloneqq t(I_a, I_b, I_c) .\nn\ee
The Coxeter number and dual Coxeter number of $\slM$ (and of $\slMhat$) are equal to $\hc$.   

\begin{lemma}\label{lem: ti}
We have the following tensor identities
\begin{alignat}{2}
 f_{ade} f_{bdg} t_{ceg} & = - \hc  t_{abc} &  t_{abc} t_{dbc}     & = 2 \frac{\hc^2-4}{\hc} \delta_{ad},\nn\\   
 f_{abc} f_{dbc}         & = - 2 \hc  \delta_{ad},\qquad &  f_{ade} t_{bgd} t_{ceg} & = \phantom{1} \frac{\hc^2-4}{\hc} f_{abc}, \nn
\end{alignat}
and the tensors $f_{a d e} f_{dfg} f_{ehi} t_{bfh} t_{cgi}$, $f_{a cd} f_{def} f_{egh} f_{fij} t_{bgi} t_{chj}$ and $f_{a cd} f_{def} f_{egh} f_{fij} t_{bhj} t_{cgi}$ are identically zero. Moreover, we have the tensor identity
\begin{equation} \label{tt identity}
t_{ea(b} t_{cd)e} = t_{e (ab} t_{cd)e}.
\end{equation}
Here $X_{(a_1 \dots a_p)} \coloneqq \sum_{\sigma \in S_p} \frac{1}{p!} X_{\sigma(a_1) \dots \sigma(a_p)}$.
\begin{proof} By direct calculation in the defining matrix representation. \end{proof}
\end{lemma}

\subsection{Quadratic and cubic states}\label{sec: S1S2def}
With $N\in \ZZ_{\geq 1}$ as above, let $z_1,\dots,z_N\in \CC$ be a collection of distinct points in the complex plane. 
For any $A\in\slMhat$ let us define the $\slMhat^{\oplus N}$-valued rational function 
\be A(z) \coloneqq \sum_{i=1}^N \frac{A^{(i)}}{z-z_i} .
\label{def: Az}\ee

\begin{lemma}\label{lem: Az relns} For any $A,B \in \slMhat$,
\begin{align} 
[A(z),B(w)] &= -\frac{ [A,B](z) - [A,B](w)}{z-w}, \qquad z\neq w, \nn\\
[A(z),B(z)] &= -[A,B]'(z) \nn \end{align}
where in the second line the ${}'$ denotes derivative with respect to the argument, $z$.\qed
\end{lemma}

Now we define states $\s_1(z)$ and $\s_2(z)$ in $\VV_0^{\bm k}$, depending rationally on $z$, as follows
\begin{align} \s_1(z) &\coloneqq \ha I^a_{-1}(z) I^a_{-1}(z) \vackk ,\nn\\
\s_2(z) &\coloneqq \ta t_{abc} I^a_{-1}(z) I^b_{-1}(z) I^c_{-1}(z) \vackk.\nn\end{align}
Define also the \emph{twist function} 
\be \varphi(z) \coloneqq \sum_{i=1}^N \frac{k_i}{z-z_i}. \label{tdef}\ee
and, for any integer $j$, the \emph{twisted derivative operator of degree $j$ with respect to $z$},
\be \DD z j \coloneqq \del_z - \frac{j}{\hc} \varphi(z).\label{def: td}\ee
Recall the translation operator $T \in \End(\VV_0^{\bm k})$.  

\subsection{Main result} We can now state the main results of the paper.
\begin{theorem}\label{thm: S1S2 zeroth prod}
For all $i,j\in\{1,2\}$, there exist $\VV_0^{\bm k}$-valued rational functions $A_{ij}(z,w)$ and $B_{ij}(z,w)$ such that 
\begin{equation} \label{SiSj 0prod}
\s_i(z)_{(0)} \s_j(w) = \big( j \DD z i - i \DD w j \big) A_{ij}(z,w) + T B_{ij}(z,w).
\end{equation}
Moreover, $A_{ij}(z,w)$ are regular at $z = w$ up to terms of the form $TZ$ with $Z \in \VV_0^{\bm k}$.
\begin{proof}
The first statement follows from a (lengthy) direct computation using the various identities from Lemma \ref{lem: ti}. Explicitly, we find that one choice of functions $A_{ij}(z,w)$ and $B_{ij}(z,w)$ is given by:
\begin{equation*}
A_{11}(z,w) = \frac{\hc}{z-w} I^a_{-2}(z) I^a_{-1}(w)\vackk, \qquad
B_{11}(z,w) = \frac{\hc}{(z-w)^2} I^a_{-1}(z) I^a_{-1}(w) \vackk
\end{equation*}
for $i=j=1$,
\begin{align*}
A_{12}(z,w) &= \frac{\hc}{2(z-w)} t_{abc} I^a_{-2}(z) I^b_{-1}(w) I^c_{-1}(w) \vackk,\\
B_{12}(z,w) &= \frac{\hc}{2 (z-w)^2} t_{abc} I^a_{-1}(z) I^b_{-1}(w) I^c_{-1}(w) \vackk,
\end{align*}
for $i=1$ and $j=2$,
\begin{align*}
A_{21}(z,w) &= \frac{\hc}{z-w} t_{abc} I^a_{-2}(z) I^b_{-1}(z) I^c_{-1}(w) \vackk,\\
B_{21}(z,w) &= \frac{\hc}{(z-w)^2} t_{abc} I^a_{-1}(z) I^b_{-1}(z) I^c_{-1}(w) \vackk,
\end{align*}
for $i=2$ and $j=1$, and finally
\begin{align} \label{A22 expression}
A_{22}(z,w) &= \frac{\hc}{2(z-w)} t_{abe}t_{cde} I^a_{-2}(z) I^b_{-1}(z) I^c_{-1}(w)I^d_{-1}(w) \vackk \nn\\
&- \frac{\hc(\hc^2 - 4)}{(z-w)^2} \bigg( \frac{\varphi(z) - \varphi(w)}{\hc} I^a_{-4}(z) I^a_{-1}(w) \nn\\
&+ \ha \big( I^a_{-4}{}'(z) I^a_{-1}(w) + I^a_{-4}(z) I^a_{-1}{}'(w) \big) - \frac{1}{\hc} f_{abc} I^a_{-3}(z) I^b_{-1}(z) I^c_{-1}(w) \bigg) \vackk \nn\\
&\quad + \frac{\hc (\hc^2 - 4)}{(z-w)^3} \big( I^a_{-4}(z) I^a_{-1}(z) - 3 I^a_{-4}(z) I^a_{-1}(w) - I^a_{-3}(z) I^a_{-2}(z) \big) \vackk
\end{align}
and
\begin{align*}
B_{22}(z,w) &= \frac{\hc}{2(z-w)^2} t_{abe}t_{cde}  I^a_{-1}(z) I^b_{-1}(z) I^c_{-1}(w)I^d_{-1}(w)\vackk \\ 
&\qquad -\frac{2\hc(\hc^2 - 4)}{(z-w)^3} \bigg( 2 \frac{\varphi(z) - \varphi(w)}{\hc} I^a_{-3}(z) I^a_{-1}(w) + I^a_{-3}(z) I^a_{-1}{}'(w)\\
&\qquad\qquad\qquad\qquad\qquad - \frac{1}{\hc} f_{abc} I^a_{-2}(z) I^b_{-1}(z) I^c_{-1}(w) \bigg) \vackk\\
&\qquad+ \frac{2 \hc (\hc^2 - 4)}{(z-w)^4} \big(I^a_{-3}(z) I^a_{-1}(z) - 5 I^a_{-3}(z) I^a_{-1}(w)\\
&\qquad\qquad\qquad\qquad\qquad - I^a_{-2}(z) I^a_{-2}(z) + \ha I^a_{-2}(z) I^a_{-2}(w)\big) \vackk.
\end{align*}
in the case $i=j=2$.

To show the `moreover' part, it suffices to expand the expressions for $A_{ij}(z,w)$ given above in $w$ near $z$ and express all singular terms in the desired form. When $i = j = 1$ we have
\begin{equation*}
A_{11}(z,w) = \frac{\hc}{2(z-w)} T I^a_{-1}(z) I^a_{-1}(z)\vackk + \ldots 
\end{equation*}
where the dots represent terms regular at $z = w$. Likewise, for $i = 1$, $j=2$ and $i=2$, $j=1$ we find, respectively,
\begin{align*}
A_{12}(z,w) &= \frac{\hc}{3(z-w)} t_{abc} T I^a_{-1}(z) I^b_{-1}(z) I^c_{-1}(z) \vackk + \ldots\\
A_{21}(z,w) &= \frac{\hc}{3(z-w)} t_{abc} T I^a_{-1}(z) I^b_{-1}(z) I^c_{-1}(z) \vackk + \ldots
\end{align*}
as required.

Now consider the case $i=j=2$. The terms of order $(z-w)^{-3}$ and $(z-w)^{-2}$ in $A_{22}(z,w)$ can be written as
\begin{align*}
&- \frac{\hc (\hc^2 - 4)}{3(z-w)^3} T \big( 2 I^a_{-3}(z) I^a_{-1}(z) + \qa I^a_{-2}(z) I^a_{-2}(z) \big) \vackk\\
&\qquad +\frac{\hc (\hc^2 - 4)}{6(z-w)^2} T \big( 5 I^a_{-3}(z) I^a_{-1}{}'(z) - I^a_{-3}{}'(z) I^a_{-1}(z) + \ha I^a_{-2}{}'(z) I^a_{-2}(z) \big) \vackk.
\end{align*}
The terms of order $(z-w)^{-1}$ in $A_{22}(z,w)$ can also be written as $T Z$ for some state $Z \in \VV_0^{\bm k}$. In order to see this, one needs to note that the singular part of the first term on the right hand side of \eqref{A22 expression} can be written as
\begin{align*}
&\frac{\hc}{z - w} t_{abe} t_{cde} I^a_{-2}(z) I^b_{-1}(z) I^c_{-1}(z) I^d_{-1}(z) \vackk\\
&\quad = \frac{\hc}{z - w} t_{abe} t_{cde} I^a_{-2}(z) I^{(b}_{-1}(z) I^c_{-1}(z) I^{d)}_{-1}(z) \vackk\\
&\quad\qquad - \frac{\hc(\hc^2 - 4)}{z - w} f_{abc} \bigg( \frac{2}{3} I^a_{-2}(z) I^c_{-2}{}'(z) I^b_{-1}(z) + \frac{1}{3} I^a_{-2}(z) I^b_{-1}(z) I^c_{-2}{}'(z) \bigg) \vackk.
\end{align*}
The first term on the right hand side is proportional to
\begin{align*}
t_{abe} t_{cde} I^a_{-2}(z) I^{(b}_{-1}(z) I^c_{-1}(z) I^{d)}_{-1}(z) \vackk 
&= t_{e(ab} t_{cd)e} I^a_{-2}(z) I^b_{-1}(z) I^c_{-1}(z) I^d_{-1}(z) \vackk\\
&= \frac{1}{4} t_{e(ab} t_{cd)e} T I^a_{-1}(z) I^b_{-1}(z) I^c_{-1}(z) I^d_{-1}(z) \vackk
\end{align*}
where in the first equality we used the identity \eqref{tt identity}.
\end{proof}
\end{theorem}

\begin{theorem} \label{thm: gsym}
For any $x \in \slM$ and $n \in \ZZ_{\geq 0}$ we have
\begin{align}
\Delta x_n \s_1(z) &= D^{(1)}_z \big( \!- \hc x_{-1}(z) \vac \delta_{n, 1}\big)\nn\\
\Delta x_n \s_2(z) &= D^{(2)}_z \big( \!- \ha\hc t_{abc} (x | I^a) I^b_{-1}(z) I^c_{-1}(z) \vac \delta_{n,1}\big).\nn
\end{align}
\begin{proof}
As for Theorem \ref{thm: S1S2 zeroth prod}, the proof is by direct calculation using the identities in Lemma \ref{lem: ti}. 
\end{proof}
\end{theorem}


\subsection{Quadratic Gaudin Hamiltonians}

Recall from \S\ref{sec: alg obs} that we are assuming the levels to be non-critical, \emph{i.e.} $k_i \neq -\hc$ for each $i \in \{ 1, \ldots, N\}$. Let
\begin{equation} \label{Seg Sug i}
\omega^{(i)} \coloneqq \frac{1}{2(k_i + \hc)} I^{a(i)}_{-1} I^{a(i)}_{-1} \vackk \in \VV_0^{\bm k}
\end{equation}
be the corresponding Segal-Sugawara state at site $i$. It has the property that for any $A = \bigotimes_{i=1}^N A_i \in \VV_0^{\bm k}$ we have $\big(\omega^{(i)}\big)_{(0)} A = A_1 \otimes \ldots \otimes A_{i-1} \otimes TA_i \otimes A_{i+1} \otimes \ldots \otimes A_N$. Let us also introduce the $\VV^{\bm k}_0$-valued rational function
\begin{equation*}
\omega(z) \coloneqq \sum_{i=1}^N \frac{\omega^{(i)}}{z - z_i}.
\end{equation*}
Consider the state
\begin{equation} \label{straight S1 def}
s_1(z) \coloneqq \s_1(z) + M \DD z1 \omega(z) \in \VV^{\bm k}_0
\end{equation}
depending rationally on $z$.

\begin{theorem} \label{thm: S1 0 cal Si}
For $i \in \{ 1, 2\}$, we have
\begin{equation*}
s_1(z)_{(0)} \s_i(w) = - \DD wi A_{1i}(z,w) + T \bigg( B_{1i}(z,w) + \DD z1 \frac{M \s_i(w)}{z-w} \bigg),
\end{equation*}
with the $\VV^{\bm k}_0$-valued rational functions $A_{1i}(z,w)$ and $B_{1i}(z,w)$ as in Theorem \eqref{thm: S1S2 zeroth prod}.
\begin{proof}
We have
\begin{align*}
\omega(z)_{(0)} \s_2(w) 
&= - \frac{1}{z-w} t_{abc} \big( I^a_{-2}(z) - I^a_{-2}(w) \big) I^b_{-1}(w) I^c_{-1}(w) \vackk\\
&= - \frac{2}{M} A_{12}(z,w) + T \bigg( \frac{\s_2(w)}{z-w} \bigg).
\end{align*}
Similarly, one finds that
\begin{equation*}
\omega(z)_{(0)} \s_1(w) 
= - \frac{1}{M} A_{11}(z,w) + T \bigg( \frac{\s_1(w)}{z-w} \bigg).
\end{equation*}
Acting with the twisted derivative $\DD z1$ on the above equations we obtain
\begin{equation*}
\big( M \DD z1 \omega(z) \big)_{(0)} \s_j(w) = - j \DD z1 A_{1j}(z,w) + T \DD z1 \frac{M \s_j(w)}{z-w},
\end{equation*}
for $j \in \{ 1,2 \}$.
The result now follows from adding this to \eqref{SiSj 0prod} with $i = 1$ and using the definition \eqref{straight S1 def} of the state $s_1(z)$.
\end{proof}
\end{theorem}

\subsection{Twisted cycles}
Let us define 
\be \P(z) \coloneqq \prod_{j=1}^N (z-z_j)^{k_j}.\nn\ee
It is a multivalued function on $\CC\setminus\{z_1,\dots,z_N\}$. 

Pick $i\in \ZZ$. Let $\gamma$ be a closed contour in $\CC\setminus\{z_1,\dots,z_N\}$ along which there exists a univalued branch of the function $\P(z)^{i/\hc}$. 
For example one can let $\gamma$ be a Pochhammer contour about any pair of the marked points $z_1,\dots,z_N$.
The pair, consisting of such a contour $\gamma$ and univalued branch of $\P(z)^{i/\hc}$ along it, defines a cycle (possibly the zero cycle) in the twisted homology corresponding to $\DD zi$ -- see \cite{LVY1} for the precise definition.

\begin{lemma} \label{lem: stokes}
One has
\begin{equation*}
\int_{\gamma} \P(z)^{-i/\hc} \DD zi f(z) dz = 0
\end{equation*}
for any meromorphic $f$ which is nonsingular on $\gamma$. 
\begin{proof}
This follows by noting that $\DD zi f(z)= \P(z)^{i/\hc} \del_z \big( \P(z)^{-i/\hc}f(z) \big)$.
\end{proof}
\end{lemma}

\subsection{Quadratic and cubic Hamiltonians}
Suppose $X,Y \in \VV_0^{\bm k}$ are any two states. Consider the formal modes $X_{(0)}, Y_{(0)} \in \tilde U_{\bm k}(\slMhat^{\oplus N})$. One can show that 
\be [X_{(0)}, Y_{(0)}] = (X_{(0)}Y)_{(0)} .\nn\ee 
This is actually a general statement which holds for any vertex algebra. See, for example, \cite[Chapter 4]{FrenkelBenZvi}.\footnote{Note that there the notation for formal modes is \emph{e.g.} $X_{[0]}$, with $X_{(0)}$ reserved for the representative in $\End\VV$.} 

It follows that, given a collection of states in $\VV_0^{\bm k}$ whose zeroth products vanish, the formal zero modes of these states define a commutative subalgebra of the algebra of observables $\tilde U_{\bm k}(\slMhat^{\oplus N})$.
In fact the zeroth products do not even need to vanish for this to be true: one can show that $(TZ)_{(0)} = 0$ for any state $Z$, so it is sufficient that the zeroth products are translates, \emph{i.e.} lie in the image of $T$. 

In our case we have the states $\s_1(z)$ and $\s_2(z)$ which depend rationally on the spectral parameter $z\in \CC\setminus\{z_1,\dots,z_N\}$. Their zeroth products, according to Theorem \ref{thm: S1S2 zeroth prod}, are zero modulo translates \emph{and} twisted derivatives in the spectral parameters. 
Let $\gamma_i$ and $\eta_i$ be any cycles of the twisted homology corresponding to $\DD zi$, for $i\in\{1,2\}$. 
In view of Lemma \ref{lem: stokes}, we have
\begin{equation} \label{com integral si sj}
\bigg[\int_{\gamma_i} \P(z)^{-i/\hc} \s_i(z)_{(0)} dz, \int_{\eta_j} \P(w)^{-j/\hc} \s_j(w)_{(0)} dw \bigg] = 0
\end{equation}
for $i,j\in \{1,2\}$, and we get a commutative subalgebra of  the algebra of observables $\tilde U_{\bm k}(\slMhat^{\oplus N})$, generated by such integrals over twisted cycles. 

However, these are not quite the Hamiltonians we want, because their action on highest weight modules is not diagonalizable. Indeed, recall that the \emph{homogeneous gradation} on $U(\slMhat^{\oplus N})$ is the $\ZZ$-gradation in which $\deg(X^{(i)}_n) = n$. 
It induces a $\ZZ_{\leq 0}$-gradation on $\VV_0^{\bm k}$, if we set $\deg(\vackk) = 0$. One can show that if $X\in \VV_0^{\bm k}$ has $\deg(X) = k$ then $\deg(X_{(n)}) = 1 + k +n$. 
Now in our case $\deg(\s_i(z)) = -i-1$, and hence
\be \deg(\s_i(z)_{(0)}) = -i \nn\ee
for $i\in\{1,2\}$. Thus the operators $\int_{\gamma_i} \P(z)^{-i/\hc} \s_i(z)_{(0)} dz\in  \tilde U_{\bm k}(\slMhat^{\oplus N})$ are of degree $-i \neq 0$ in the homogeneous gradation. Consequently if $M$ is any graded module over $\tilde U_{\bm k}(\slMhat^{\oplus N})$ whose subspace of grade $n$ is trivial for sufficiently large $n$ then the operator $\int_{\gamma_i} \P(z)^{-i/\hc} \s_i(z)_{(0)} dz$ has no non-zero generalised eigenvalue in $M$. 

Instead, what we want are commuting operators in $\tilde U_{\bm k}(\slMhat^{\oplus N})$ that are of degree zero in the homogeneous gradation. There is a general procedure to construct such operators, which roughly speaking involves a change of coordinate\footnote{(in the ``loop variable'', $t$, not the ``spectral variable'' $z$)} to go from the plane to the cylinder \cite{Nahm}.  For our present purposes we just recall the following facts: there is a notion of what we shall call the \emph{Fourier modes}, $X_{\fm{n}}\in\tilde U_{\bm k}(\slMhat^{\oplus N})$, of any state $X\in \VV_0^{\bm k}$.
These modes have the property that $X_{\fm{n}}$ always has grade $n$. 
They obey the commutator formula
\be [X_{\fm{m}},Y_{\fm{n}}]  
   = \left(X_{(0)} Y + \sum_{k=1}^\8 \frac{m^k X_{(k)} Y}{k!} \right)_{\fm{m+n}} \label{fmcom}\ee
for any states $X,Y\in \VV_0^{\bm k}$ and any $n,m\in \ZZ$. 
In particular 
\be [X_{\fm{0}},Y_{\fm{0}}] = (X_{(0)}Y)_{\fm{0}}.\nn\ee 
Also, $(TZ)_{\fm{0}} = 0$ for any state $Z\in \VV_0^{\bm k}$. 

One has $(x_{-1}^{(i)}\vackk)_{\fm{n}} = x_n$ for $x\in \slM$ and there is a formula, analogous to the normal ordered product formula \eqref{nord}, which allows one to compute by recursion the Fourier modes of a general state $X\in \VV_0^{\bm k}$. 
(See for example equation (7) in \cite{BLZ}.)
In general the resulting expressions for modes of states are rather intricate but for $\s_1(z)_{\fm{0}}$ and $\s_2(z)_{\fm{0}}$ they are simple: one finds
\begin{align} 
\s_1(z)_{\fm{0}} &=  \frac{\dim(\slM)} {24} \varphi'(z) + \ha I^a_0(z) I^a_0(z) + \sum_{n>0} I^a_{-n}(z) I^a_n(z), \label{S10}\\
\s_2(z)_{\fm{0}} &= \ta t_{abc} \sum_{j,k \geq 0} \Bigl( I^a_{-1-k}(z) I^b_{-1-j}(z) I^c_{2+j+k}(z) + 2I^a_{-1-k}(z) I^b_{1+k-j}(z) I^c_{j}(z) \nn\\
& \hspace{170pt} + I^a_{-j-k}(z) I^b_{k}(z) I^c_{j}(z) \Bigr). \label{S20}
\end{align}  

Let us now define 
\begin{equation} \label{def: Hig}
\hat Q_i^{\gamma} \coloneqq \int_\gamma \P(z)^{-i/\hc} \s_i(z)_{\fm{0}} dz
\end{equation}
for $i\in\{1,2\}$ and for $\gamma$ any cycle of the twisted homology corresponding to $\DD zi$.

Recall the Lie algebra $\g \coloneqq \slMhat \oplus \CC\cocent$ from \S\ref{sec: loop rel slM}. Let $\tilde U(\g^{\oplus N}) \coloneqq \varprojlim U(\g^{\oplus N})/ J^N_n$ be the completion of $U(\g^{\oplus N})$ where for each $n \geq 0$, $J^N_n$ is the ideal of $U(\g^{\oplus N})$ generated by $a^{(i)}_r$ for all $r \geq n$, $a \in \slM$ and $i \in \{ 1, \ldots, N \}$. The quadratic Gaudin Hamiltonians are the elements of $\tilde U(\g^{\oplus N})$ defined as
\begin{equation*}
\mathcal H_i \coloneqq \sum_{\substack{j=1\\ j \neq i}}^N \frac{\cent^{(i)} \cocent^{(j)} + \cocent^{(i)} \cent^{(j)} + \sum_{n\in \ZZ} I^{a(i)}_{-n} I^{a(j)}_n}{z_i - z_j}, \qquad i = 1, \ldots, N.
\end{equation*}
For each $i \in \{ 1, \ldots, N \}$ we also have the $i^{\rm th}$ copy of the quadratic Casimir of $\g$ in $\tilde U(\g^{\oplus N})$, which is defined as
\begin{equation*}
\mathcal C^{(i)} \coloneqq (\cent^{(i)} + M) \cocent^{(i)} + \ha I^{a(i)}_0 I^{a(i)}_0 + \sum_{n > 0} I^{a (i)}_{-n} I^{a(i)}_n.
\end{equation*}
The algebra $\tilde U_{\bm k}(\slMhat^{\oplus N})$, introduced in \S\ref{sec: alg obs}, is isomorphic to the quotient of $\tilde U(\g^{\oplus N})$ by the ideal generated by $\cent^{(i)} - k_i$ and $\mathcal C^{(i)}$ for each $i \in \{ 1, \ldots, N\}$.

Recall the state $s_1(z) \in \VV^{\bm k}_0$ defined in \eqref{straight S1 def}.
\begin{lemma}
The operator $s_1(z)_{\fm{0}}$ is the image in $\tilde U_{\bm k}(\slMhat^{\oplus N})$ of the operator
\begin{equation*}
\sum_{i=1}^N \frac{\mathcal C^{(i)}}{(z - z_i)^2} + \sum_{i=1}^N \frac{\mathcal H_i}{z - z_i} \in \tilde U(\g^{\oplus N}).
\end{equation*}
\begin{proof}
Computing the Fourier zero mode of the state $\s_1(z)$ we find
\begin{align*}
\s_1(z)_{\fm{0}} &= \sum_{i=1}^N \frac{1}{(z-z_i)^2} \bigg( \!\! - \frac{\dim(\slM)}{24} k_i + \ha I^{a(i)}_0 I^{a(i)}_0 + \sum_{n > 0} I^{a(i)}_{-n} I^{a(i)}_n \bigg)\\
&\qquad\qquad\qquad\qquad\qquad + \sum_{i=1}^N \frac{1}{z-z_i} \sum_{\substack{j=1\\ j\neq i}}^N \frac{1}{z_i - z_j} \sum_{n \in \ZZ} I^{a(i)}_{-n} I^{a(j)}_n.
\end{align*}
On the other hand, we also have
\begin{align*}
\hc \DD z1 \omega(z)_{\fm{0}} = - \sum_{i=1}^N \frac{(k_i + \hc) \big( \omega^{(i)} \big)_{\fm{0}}}{(z-z_i)^2} - \sum_{i=1}^N \frac{1}{z-z_i} \sum_{\substack{j=1\\ j\neq i}}^N \frac{k_i \big( \omega^{(j)} \big)_{\fm{0}} + k_j \big( \omega^{(i)} \big)_{\fm{0}}}{z_i - z_j}.
\end{align*}

The result now follows from combining the above and noting that the Fourier zero mode of the Segal-Sugawara state $\omega^{(i)}$, for each $i \in \{ 1, \ldots, N\}$, as defined in \eqref{Seg Sug i}, reads
\begin{equation*}
\big( \omega^{(i)} \big)_{\fm{0}} = \frac{\ha I^{a(i)}_0 I^{a(i)}_0 + \sum_{n > 0} I^{a(i)}_{-n} I^{a(i)}_n}{k_i + M} - \frac{k_i \dim(\slM)}{24 (k_i+ M)}. \qedhere
\end{equation*}
\end{proof}
\end{lemma}

We now have the following corollary of Theorems \ref{thm: S1S2 zeroth prod}, \ref{thm: gsym} and \ref{thm: S1 0 cal Si}.

\begin{theorem}\label{thm: main}
The operators $\hat Q_i^\gamma\in \tilde U_{\bm k}(\slMhat^{\oplus N})$ defined in \eqref{def: Hig} are of grade 0 in the homogeneous gradation. They have the following properties:
\begin{itemize}
  \item[$(i)$] For $i,j\in \{1,2\}$ and any cycles $\gamma$ and $\eta$ of the twisted homology corresponding to $\DD zi$ and $\DD zj$, respectively, we have
\begin{equation*}
\big[ \hat Q_i^\gamma, \hat Q_j^\eta  \big] = 0.
\end{equation*}
  \item[$(ii)$] For all $i \in \{ 1, 2 \}$, any cycle $\gamma$ of the twisted homology corresponding to $\DD zi$ and $j \in \{ 1, \ldots, N \}$ we have
\begin{equation*}
[\mathcal H_j, \hat Q^\gamma_i] = 0.
\end{equation*}
  \item[$(iii)$] For all $i \in \{ 1, 2 \}$, any cycle $\gamma$ of the twisted homology corresponding to $\DD zi$ and all $x\in \slM$ and $n\in \ZZ$, we have
\begin{equation*}
\pushQED{\qed} 
[\Delta x_n, \hat Q_i^\gamma ] = 0. \qedhere
\popQED
\end{equation*}
\end{itemize}
%
%
\end{theorem}

According to the conjecture in \cite{LVY1}, the Hamiltonians $\hat Q_i^{\gamma}$ can be diagonalized by means of the Bethe ansatz and have certain specific eigenvalues encoded by affine opers. For $\hat Q_1^\gamma$ this is known; see \cite[\S5]{LVY1}. In the next two sections we perform a simple Bethe ansatz calculation to check that  the conjecture is correct also for $\hat Q_2^\gamma$ at least for zero and one Bethe roots.

\section{Opers and eigenvalues} \label{sec: op and eval}

When $\g$ is of finite type, it is well known \cite{Fre95, Fopersontheprojectiveline, Fopers} that the spectrum of Hamiltonians for the Gaudin model associated with $\g$ are described in terms of \emph{opers} for the Langlands dual Lie algebra $\lg$, \emph{i.e.} the Lie algebra with transposed Cartan matrix. When $\g$ is of affine type, a series of conjectures was made in \cite{FFsolitons} regarding the description of the spectrum of both local and non-local Hamiltonians for the affine Gaudin model in terms of $\lg$-opers, where $\lg$ is again the Langlands dual of $\g$ with transposed Cartan matrix. Building on the proposal of \cite{FFsolitons}, in \cite{LVY1} we gave an explicit conjecture of how such $\lg$-opers encode the eigenvalues of the \emph{local} higher Hamiltonians. In this section we recall the notion of (Miura) opers from \cite{Fopersontheprojectiveline, FFsolitons} and the corresponding notion of quasi-canonical form introduced in \cite{LVY1}, focusing on the special case $\g = \slMhat$.

Since $\slMhat$ has symmetric Cartan matrix it is self-dual: ${}^L\slMhat \cong \slMhat$. Nevertheless, to keep the general structure in sight and to follow the notation of \cite{LVY1} we prefer to  distinguish the two copies of $\slMhat$. Thus, let $\check e_i$, $\check f_i$ be the Chevalley-Serre generators of the dual copy $\lsln$. They obey the same relations \eqref{KM relations} as the generators $e_i$, $f_i$ above but with the roots and coroots $\alpha_i$ and $\check\alpha_i$ interchanged. The Cartan subalgebra of $\lsln$ is $\lh \coloneqq \h^*$. Let $\lsln=\ln_- \oplus \lh \oplus \ln_+$ be the Cartan decomposition.

\subsection{Opers and Miura opers of type $\lsln$} 
\label{sec: opers}
Pick a derivation element $\rho\in \lh$ such that $[\rho, \check e_i] = \check e_i$, $[\rho,\check f_i] = -\check f_i$ for each $i=0,1,\dots,{M-1}$. Explicitly,
one has
\be \rho \coloneqq M\cocent + \dot\rho \label{def: rho}\ee
where $\dot\rho\coloneqq \sum_{i=1}^{M-1} \omega_i\in \dot\h^*$ is the Weyl vector of $\slM$.

Let $p_{-1} \coloneqq \sum_{i=0}^{{M-1}} \check f_i\in \ln_-$. A \emph{Miura oper} of type $\lsln$ is a connection of the form
\begin{subequations}\label{Mop}
\be d + \left(p_{-1} + u(z)\right) dz \ee 
where $u(z)$ is a meromorphic function valued in $\lh=\h^*$. 
Since $\{\alpha_i\}_{i=0}^{M-1} \cup \{\rho\}$ is a basis of $\lh$ we have 
\be u(z) = \sum_{i=0}^{M-1} u_i(z) \alpha_i - \frac{\varphi(z)}{\hc}\rho \ee
\end{subequations}
for some meromorphic functions $\{u_i(z)\}_{i=0}^{M-1}$ and $\varphi(z)$.
An $\lsln$-\emph{oper} is a gauge equivalence class of connections of the form 
\be d + (p_{-1} + b(z)) dz, \nn\ee
where $b(z)$ is a meromorphic function valued in $\lb_+ = \lh \oplus \ln_+$, under the gauge action of the group ${}^L\!N_+ = \exp(\ln_+)$. 
Any $\lsln$-oper has a \emph{quasi-canonical} representative \cite{LVY1}: namely, by suitable gauge transformation the connection can be brought to the form
\be
\nabla = d + \Bigg( p_{-1} - \frac{\varphi}{\hc} \rho + \sum_{j \in E} v_j p_j  \Bigg) dz,\label{qcf}
\ee
where the sum is over the set
\begin{equation*}
E = \ZZ_{\geq 1} \setminus \hc \ZZ_{\geq 1}
\end{equation*}
of positive \emph{exponents} of $\lsln$, and where for each $j\in E$, $v_j$ is a meromorphic function. For each $j\in \pm E$, $p_j$ is a certain element of $\ln_+$ of grade $j$ in the principal gradation, \emph{i.e.}
\be [\rho, p_j] = j\, p_j, \nn\ee
such that the following commutation relations hold with $p_{-1}$:
\be  [p_{-1},p_j]=\begin{cases}  \,\,\,\,0 & j\in \pm E\setminus\{1\} \\  - \delta & j=1, \end{cases}\nn\ee
where $\delta \coloneqq \sum_{i=0}^{M-1} \alpha_i$ is central in $\lsln$.
It can be shown that, for each $j\in \pm E$, a non-zero such $p_j$ exists and is unique up to normalization.
Let $\a \coloneqq \CC\{p_j\}_{j\in \pm E} \oplus \CC \delta \oplus \CC \rho \subset \lsln$. Then $\a$ is a Lie subalgebra of $\lsln$ called the \emph{principal subalgebra} (the remaining non-trivial commutation relations are $[p_m,p_n] = m\delta_{m+n,0}\delta$). 
The normalization of the $p_j$ can be so chosen that the restriction to $\a$ of the bilinear form $(\cdot | \cdot)$ is given by
\begin{equation*}
(\delta | \rho) = (\rho | \delta) = \hc, \qquad (p_m | p_n) = \hc \delta_{m+n,0}, \qquad m, n \in \pm E.
\end{equation*}

The expression for $v_1$ below can be found in \cite[Proposition 3.10]{LVY1}.

\begin{proposition}\label{prop: v1v2}
The coefficients of $p_1$ and $p_2$ of any quasi-canonical form of the $\lsln$-oper defined by the Miura $\lsln$-oper \eqref{Mop} are given by
\begin{align*}
v_1 & = \frac{1}{\hc} \big( \ha (u|u) + \DD z 1 (\rho | u) \big),\\
v_2 & = \frac{1}{\hc} \bigg( \!\! - \sum_{i=0}^{M-1} u_i (u^2_{i+1} - u^2_{i-1})   -\ha \sum_{i=0}^{M-1} u_i (u'_{i+1} - u'_{i-1}) + \DD z 2 f_2 \bigg),
\end{align*}
where $f_2$ is an arbitrary meromorphic function. 
\end{proposition}
\begin{proof}
Let $g(z) = \exp(m(z)) \in {}^L\! N_+$ with $m$ a meromorphic function valued in $\ln_+$. Let $m = \sum_{n=1}^\8m_n$ be the decomposition of $m$ in the principal gradation.  

Let $\nabla_{\partial_z} = \del_z + p_{-1} + u - \frac{\varphi}{\hc} \rho$, as in \eqref{Mop}.
We demand that
\begin{equation} \label{gauge tr quasi-can}
g \nabla_{\partial_z} g^{-1} = \del_z+ p_{-1} - \frac{\varphi}{\hc} \rho + v_1 p_1 + v_2 p_2 + \dots
\end{equation}
where the dots denote terms of degree at least three in the principal gradation.
In grade 0 we have $u+[m_1, p_{-1}] -  \frac{\varphi}{\hc} \rho = - \frac{\varphi}{\hc} \rho$ and hence
$u = - [m_1,p_{-1}]$.
Yet since $u = \sum_{i=0}^{M-1} u_i \alpha_i$, we must set 
\begin{equation*}
m_1 = - \sum_{i=0}^{M-1} u_i\check e_i.
\end{equation*}

Now consider the component of equation \eqref{gauge tr quasi-can} in grade 1. We need
\begin{equation} \label{g1e}
[m_2,p_{-1}] + \ha [m_1,[m_1,p_{-1}]] + [m_1,u] + \frac{\varphi}{\hc} m_1 - m_1' = v_1p_1.
\end{equation}
Consider taking the overlap $(p_{-1}|\cdot)$ of both sides of this equation. Recall the definition of $\DD z j$ from \eqref{def: td}. Since $(p_{-1}|p_1) = \hc$ we find, using the invariance of the bilinear form, that
\begin{align}
\hc v_1 &= -\ha \big( [m_1,p_{-1}] \big| [m_1, p_{-1}] \big) - \big( [m_1,p_{-1}] \big| u \big) - \DD z 1(p_{-1}|m_1) \nn\\
&= \ha (u|u)  
+ \DD z 1 (\rho|u),
\nn\end{align}  
where $(p_{-1}|m_1) = ([p_{-1},\rho]|m_1) = (\rho|[m_1,p_{-1}]) = - (\rho|u)$.  
(The equation \eqref{g1e} also fixes the component of $m_2$ in the orthogonal complement $\c_2$ of $\a_2$, but we won't need that here.) 

We turn to the component of equation \eqref{gauge tr quasi-can} in grade 2. We have
\begin{align*}
(g\del_z g^{-1})_2 &= -m_2' - \ha [m_1,m_1'],\\
(gp_{-1} g^{-1})_2 &= 
      [m_3,p_{-1}] 
+ \ha [m_1,[m_2,p_{-1}]] 
+ \ha [m_2,[m_1,p_{-1}]] 
+ \sa [m_1,[m_1,[m_1,p_{-1}]]]  \nn\\
&=    [m_3,p_{-1}] 
+ \ha [[m_1,m_2],p_{-1}] 
-  [m_2,u ] 
- \sa [m_1,[m_1,u ]],\\
(gug^{-1})_2 &= [m_2,u] + \ha [m_1,[m_1,u]],\\
- (g \rho g^{-1})_2 &= -[m_2,\rho] - \ha [m_1,[m_1,\rho]] = 2 m_2.
\end{align*}
Hence we need
\begin{equation*}
\big[ m_3 + \ha [m_1,m_2], p_{-1} \big] + \ta [m_1,[m_1,u]] - \ha [m_1,m_1'] - \DD z 2 m_2 = v_2 p_2.
\end{equation*}
Taking $(p_{-2}|\cdot)$ of both sides and then using $(p_{-2} | p_2) = \hc$, the invariance of $(\cdot|\cdot)$ and the relation $[p_{-1},p_{-2}]=0$, we find
\begin{equation*}
\hc v_2 = - \ta \left( \left[m_1,p_{-2}\right]\middle| \left[m_1,u\right] \right) + \ha \left(\left[m_1,p_{-2}\right]\middle| m_1' \right) - \DD z 2 (p_{-2} | m_2).
\end{equation*} 
Now in fact, in type $A$, we have $p_{-2} = -\sum_{i=0}^{M-1}[\check f_i,\check f_{i+1}]$ (addition of indices modulo $\hc$). Hence
\begin{equation*}
[m_1,p_{-2}] = \sum_{i,j=0}^{\hc - 1} u_i [ \check e_i ,[\check f_j,\check f_{j+1}]] = - \sum_{i=0}^{M-1} u_i (\check f_{i-1} - \check f_{i+1})
\end{equation*}
and 
\begin{equation*}
[m_1,u] = - \sum_{i,j=0}^{M-1} u_i u_j [\check e_i, \alpha_j] = \sum_{i,j=0}^{M-1} a_{ij} u_i u_j \check e_i.
\end{equation*}
Since the coefficient of $p_2$ of an $\lsln$-oper in quasi-canonical form is defined only up to twisted derivative, we therefore have
\begin{align*}
\hc v_2 &= \ta \sum_{i,j,k=0}^{M-1} a_{ij} (\delta_{i,k-1} - \delta_{i,k+1}) u_i u_j u_k 
   - \ha \sum_{i,j=0}^{M-1}  (\delta_{i,j-1} - \delta_{i,j+1}) u_i u'_j + \DD z 2 f_2\\
&= - \sum_{i=0}^{M-1} u_i (u^2_{i+1} - u^2_{i-1})   -\ha \sum_{i=0}^{M-1} u_i (u'_{i+1} - u'_{i-1})    
              + \DD z 2 f_2,
\end{align*}
for an arbitrary meromorphic function $f_2$. Here we used the fact that the Chevalley-Serre generators of $\lsln$ are normalised as $(\check e_i | \check f_j) = \delta_{ij}$.
\end{proof}


\section{Bethe Ansatz for 0 and 1 roots}\label{sec: BA}

\subsection{Homogeneous vs. principal gradations}
Let us write $\Q(z)\in \tilde U_{\bm k}(\slMhat^{\oplus N})$ for the Fourier zero mode of the state $\s_2(z)$ as in \eqref{S20},
\begin{multline}
\Q(z) \coloneqq \s_2(z)_{\fm{0}} = \ta t_{abc} \sum_{j,k \geq 0} \Bigl( I^a_{-1-k}(z) I^b_{-1-j}(z) I^c_{2+j+k}(z)\\  + 2I^a_{-1-k}(z) I^b_{1+k-j}(z) I^c_{j}(z)) 
 + I^a_{-j-k}(z) I^b_{k}(z) I^c_{j}(z) \Bigr) .\nn \end{multline}
By definition of $\tilde U_{\bm k}(\slMhat^{\oplus N})$, $\Q(z)$ is an infinite sum
\begin{equation*}
\Q(z) = \Q(z)^\homog_0 + \Q(z)^\homog_1 + \Q(z)^\homog_2 + \dots
\end{equation*}
such that, for each $n\geq 1$ we have $\Q(z) = \Q(z)^\homog_0 + \dots + \Q(z)^\homog_{n-1}$ modulo $J_{n}^N$. Recall that $J_n^N$ is the left ideal in $U_{\bm k}(\slMhat^{\oplus N})$ generated by 
elements of grade at least $n$ in the homogeneous gradation. 
Let $J_{n}^{N,\princ}$ denote the left ideal generated by 
elements of grade at least $n$ in the principal gradation. For each $n \geq 1$, there is a nesting of ideals $J_{\hc n}^{N,\princ} \subset J_n^{N}\subset J_{\hc n - \hc + 1}^{N,\princ}$.
So we could equivalently have defined $\tilde U_{\bm k}(\slMhat^{\oplus N})$ as the completion with respect to the inverse system $U_{\bm k}(\slMhat^{\oplus N}) \big/ J_n^{N,\princ}$.

Let $\b_- \coloneqq \n_- \oplus \h$. Denote by $U_{\bm k}(\b_-^{\oplus N})$ the quotient of $U(\b_-^{\oplus N})$ by the two-sided ideal generated by $\cent^{(i)} - k_i$ for all $i \in \{ 1, \ldots, N \}$, cf. \S\ref{sec: alg obs}. We have the vector space isomorphism $U_{\bm k}(\slMhat^{\oplus N}) \cong_\CC U_{\bm k}(\b_-^{\oplus N}) \otimes U(\n_+^{\oplus N})$. Let $U(\n_+^{\oplus N})_n$ denote the subspace of $U(\n_+^{\oplus N})$ spanned by elements of total degree $n$ in the principal gradation. We have the unique decomposition
\begin{equation} \label{Qsum}
\Q(z) = \Q(z)_0 + \Q(z)_1 + \Q(z)_2 + \ldots
\end{equation}
where $\Q(z)_n \in U_{\bm k}(\b_-^{\oplus N}) \otimes U(\n_+^{\oplus N})_n$.

Recall the definition \eqref{def: rho} of the derivation element $\rho$ which measures the grade in the principal gradation. Let $\check\eta_i\in \h$ be the set of fundamental coweights adapted to the principal gradation; namely, such that $\{\check\eta_i\}_{i=0}^{\hc-1} \cup \{ \cent\}\subset \h$ is the dual basis to  $\{\alpha_i\}^{\hc-1}_{i=0}\cup \{\rho/\hc\}\subset \h^*$:
\be \langle\alpha_i,\check\eta_j\rangle = \delta_{ij}, \quad \langle \rho, \check\eta_j \rangle = 0, \quad \langle \alpha_i, \cent \rangle=0, \quad \langle \rho, \cent \rangle =\hc \nn\ee
for all $i,j\in \{0,1,\dots,{\hc-1}\}$.
\begin{proposition}\label{prop: Q0Q1}
The first two terms in the sum \eqref{Qsum} are 
\begin{align}
\Q(z)_0 &= 
- \sum_{i,j,k=0}^{\hc-1} \check\eta_i(z) \big( \check\eta_{i+1}(z)^2 - \check\eta_{i-1}(z)^2 \big) + \ha  \sum_{i,j=0}^{\hc-1} \check\eta_i(z) \left( \check\eta_{i+1}'(z) - \check\eta_{i-1}'(z) \right),\nn\\
\Q(z)_1 &= \sum_{i,j=0}^{\hc-1} f_i(z) \left(\check \eta_{i+1}(z)- \check\eta_{i-1}(z) \right) e_i(z),  
\end{align} 
modulo twisted derivatives of degree 2.
\end{proposition}
\begin{proof} By direct calculation.\end{proof}

\subsection{Spin chain $\MM_{\bm \lambda}$}
For each $i\in \{1,\dots,N\}$ let $\lambda_i\in \h^*$ be a weight of $\slMhat$ such that
\be \langle \lambda_i,\cent \rangle = k_i .\label{lk}\ee
Let $\MM_{\lambda_i}$ denote the Verma module over $\g$ of highest weight $\lambda_i$ and define 
\be \MM_{\bm \lambda} \coloneqq \MM_{\lambda_1} \ox \dots \ox \MM_{\lambda_N}.\nn\ee
It is a smooth module over $\slMhat^{\oplus N}$ of level $\bm k$, and therefore a module over the algebra of observables $\tilde U_{\bm k}(\slMhat^{\oplus N})$. 
Let $\vacl$ denote the highest weight vector in $\MM_{\bm \lambda}$.

\subsection{Vacuum eigenvalues}
For this subsection, consider the Miura $\lsln$-oper \eqref{Mop}, with
\be u(z)  = u^{(0)}(z) \coloneqq - \sum_{j=1}^N \frac{\lambda_j}{z-z_j}.\nn\ee
That is, $u_i(z) = - \sum_{j=1}^N \frac{\langle\lambda_j,\check \eta_i\rangle}{z-z_j}$ for each $i\in \{0,1,\dots,{\hc-1}\}$, and $\varphi(z)$ is the twist function as in \eqref{tdef}, by virtue of \eqref{lk}. 
Let $v_1(z)$ and $v_2(z)$ be the coefficients of a quasi-canonical form of the underlying $\lsln$-oper, as in Proposition \ref{prop: v1v2}. 
\begin{proposition}
Up to twisted derivatives of degree 2, the vacuum eigenvalue of $\Q(z)$ is $- \hc v_2(z)$, \emph{i.e.}
\begin{equation*}
\Q(z) \vacl = - \hc v_2(z) \vacl + \DD z2 | \varepsilon(z) \rangle
\end{equation*}
for some vector $| \varepsilon(z) \rangle \in \MM_{\bm \lambda}$ depending rationally on $z$. In particular, for any cycle $\gamma$ in the twisted homology corresponding to $\DD z2$ we have
\begin{equation*}
\hat Q^\gamma_2 \vacl = - \hc \int_\gamma \P(z)^{-2/\hc} v_2(z) dz \; \vacl.
\end{equation*}
\end{proposition}
\begin{proof} We have $\Q(z)_n\vacl = 0$ for all $n\geq 1$, \emph{i.e.} $\Q(z) \vacl = \Q(z)_0\vacl$. By definition of the $\check{\eta}_i$, 
\begin{equation*}
\check \eta_i(z) \vacl = - \langle u(z), \check \eta_i \rangle \vacl = - u_i(z) \vacl
\end{equation*}
and then the statement follows from Propositions \ref{prop: v1v2} and \ref{prop: Q0Q1}.
\end{proof}
\subsection{Eigenvalues at depth 1}
Let $w\in \CC\setminus\{z_1,\dots,z_N\}$ be an additional point in the complex plane, distinct from the marked points $z_1,\dots,z_N$. Pick $n\in\{0,1,\dots, \hc-1\}$. Now, for this subsection, we consider the Miura $\lsln$-oper, \eqref{Mop}, with
\begin{align} u(z) &= u^{(0)}(z) + \frac{\alpha_n}{z-w} 
= - \sum_{j=1}^N \frac{\lambda_i}{z-z_i} + \frac{\alpha_n}{z-w}\label{u1}\end{align}
and let $v_1(z)$ and $v_2(z)$ be the coefficients of a quasi-canonical form of its underlying $\lsln$-oper, as in Proposition \ref{prop: v1v2}.

The \emph{weight function} or \emph{Schechtman-Varchenko vector} for one Bethe root of colour $n\in\{0,1,\dots,{M-1}\}$ is by definition $f_n(w) \vacl$, where
\begin{equation*}
f_n(w) = \sum_{j=1}^N \frac{f_n^{(i)}}{w-z_i}
\end{equation*}
using the notation \eqref{def: Az}.

\begin{proposition}
Suppose $w$ obeys the Bethe equation
\begin{equation*}
0= \sum_{j=1}^N \frac{\langle \lambda_j, \check\alpha_n\rangle}{w-z_j}.
\end{equation*}
Then, up to twisted derivatives of degree 2, the weight function $f_n(w) \vacl$ is an eigenstate of $\Q(z)$ with eigenvalue $- \hc v_2(z)$, \emph{i.e.}
\begin{equation*}
\Q(z) f_n(w) \vacl = - \hc v_2(z) f_n(w) \vacl + \DD z2 |\tilde \varepsilon(z) \rangle,
\end{equation*}
for some vector $| \tilde \varepsilon(z) \rangle \in \MM_{\bm \lambda}$ depending rationally on $z$.
In particular, for any cycle $\gamma$ in the twisted homology corresponding to $\DD z2$ we have
\begin{equation*}
\hat Q^\gamma_2 f_n(w) \vacl = - \hc \int_\gamma \P(z)^{-2/\hc} v_2(z) dz \; f_n(w) \vacl.
\end{equation*}
\end{proposition}
\begin{proof} 
By definition $[\check\eta_i, f_n] = -f_n \langle \alpha_n,\check\eta_i\rangle = -\delta_{in} f_n$.
Using this and Lemma \ref{lem: Az relns}, one computes, starting from the expressions for $\Q(z)_0$ in Proposition \ref{prop: Q0Q1},
\begin{align*}
&[\Q(z)_0, f_n(w) ] \vacl = \Bigg( \!- \sum_{i=0}^{M-1} a_{ni} u^{(0)}_i(z) \frac{u^{(0)}_{n+1}(z) -  u^{(0)}_{n-1}(z)}{z-w}\\
&\qquad\qquad\qquad\qquad\qquad - \frac{u^{(0)}_{n+1}(z) - u^{(0)}_{n-1}(z) }{2 (z-w)^2} + \frac{u^{(0)'}_{n+1}(z) - u^{(0)'}_{n-1}(z) }{2(z-w)}\Bigg) f_n(w) \vacl\\
&\qquad\quad + \left(\sum_{i=0}^{M-1} a_{ni} u^{(0)}_i(z) \frac{u^{(0)}_{n+1}(z) -  u^{(0)}_{n-1}(z)}{z-w}
    - \del_z   \frac{u^{(0)}_{n+1}(z) - u^{(0)}_{n-1}(z) }{z-w}\right) f_n(z) \vacl.
\end{align*}
Here one recognizes the first term on the right hand side as the required correction to the vacuum value of $- \hc v_2(z)$, cf. \eqref{u1} and Proposition \ref{prop: v1v2}. Next, one finds
\begin{align*}
[\Q(z)_1,f_n(w)] \vacl & = \sum_{i=0}^{\hc-1} \left[f_i(z) \left(\check \eta_{i+1}(z)- \check\eta_{i-1}(z) \right) e_i(z), f_n(w) \right] \vacl\\
&= - f_n(z) \left(\check \eta_{n+1}(z)- \check\eta_{n-1}(z) \right) \frac{\check\alpha_n(z) - \check\alpha_n(w)}{z-w} \vacl\\
&= - f_n(z) \vacl \big(u^{(0)}_{n+1}(z) - u^{(0)}_{n-1}(z)\big) \frac{ \langle u^{(0)}(z),\check\alpha_n\rangle  - \langle u^{(0)}(w) , \check\alpha_n\rangle }{z-w}.
\end{align*}
Now we have
\begin{equation*}
\langle u(z)^{(0)}, \check\alpha_n\rangle = - \frac{\varphi(z)}{\hc} \langle \rho,\check\alpha_n\rangle + \sum_{i=0}^{\hc-1} \langle \alpha_i,\check\alpha_n\rangle  u^{(0)}_i(z) = - \frac{\varphi(z)}{\hc} + \sum_{i=0}^{\hc-1} a_{ni} u^{(0)}_i(z).
\end{equation*}
Hence, one arrives at
\begin{align*}
\Q(z) f_n(w) \vacl &= - \hc v_2(z) f_n(w) \vacl
- \DD z 2 \Big( \ha \big( u^{(0)}_{n+1}(z) - u^{(0)}_{n-1}(z) \big) \frac{f_n(z)}{z-w} \vacl \Big)\\
&\qquad\qquad\qquad + \langle u^{(0)}(w) ,\check\alpha_n\rangle \big( u^{(0)}_{n+1}(z) - u^{(0)}_{n-1}(z) \big) \frac{f_n(z)}{z-w} \vacl.
\end{align*}
The second line on the right hand side vanishes by the Bethe equation. 
\end{proof}

\section{Two-point case, coset construction and the W3 algebra}\label{sec: two-point}


In this section we specialize to the case of $N=2$ marked points. For convenience, let us choose them to be $z_1 = 0$ and $z_2=1$. 
Let $\gamma$ denote a Pochhammer contour around these two points.  
For $a\in \CC$ and $n\in \ZZ_{\geq 0}$ let $(a)_n$ denote the falling factorial
\be \ff a n \coloneqq \prod_{k=0}^{n-1} (a-k) = a(a-1)\dots (a-n+1)\nn\ee
with $(a)_0 \coloneqq 1$. 
For $a,b\in \CC$, define 
\be B(a,b) \coloneqq \int_\gamma z^a (z-1)^b dz .\nn\ee
\begin{lemma} \label{lem: beta}
\be B(a-1,b) = \frac{a+b+1}{a} B(a,b), \qquad B(a,b-1) = -\frac{a+b+1}{b} B(a,b). \nn\ee 
\end{lemma}
\begin{proof} Since the contour $\gamma$ has no boundary, $aB(a-1,b) + bB(a,b-1) = 0$. From the identity $\frac{1}{z(z-1)} - \frac{1}{z-1} + \frac 1 z=0$, one sees that $B(a-1,b-1) - B(a,b-1) + B(a-1,b) = 0$. The result follows. \end{proof} 

Define
\begin{align*}
\T &\coloneqq \T^{(1)} + \T^{(2)} - \T^{(diag)}\\
&\, = \frac{1}{2(k_1+\hc)} I^{a(1)}_{-1}  I^{a(1)}_{-1} \vackk  +\frac{1}{2(k_2+\hc)} I^{a(2)}_{-1}  I^{a(2)}_{-1} \vackk \\
&\qquad\qquad\qquad\qquad -  \frac{1}{2(k_1+k_2+\hc)} (I^{a(1)}_{-1}+I^{a(2)}_{-1}) (I^{a(1)}_{-1}+I^{a(2)}_{-1})   \vackk.
\end{align*}
Define also the state $\W \in  \VV_0^{\bm k}$ by 
\begin{align*}
\W &\coloneqq C(k_1,k_2) \bigg(\ta t_{abc} I^{a(1)}_{-1} I^{b(1)}_{-1} I^{c(1)}_{-1} \vackk  \left(-\aaa 2\right) \left(-\aaa 2 - 1\right) \left(-\aaa 2 -2\right) \\
&\quad\qquad\qquad\,\,\,\,{}- t_{abc} I^{a(1)}_{-1} I^{b(1)}_{-1} I^{c(2)}_{-1} \vackk  \left(-\aaa 1 -2\right) \left(-\aaa 2 - 1\right) \left(-\aaa 2 -2\right) \\
&\quad\qquad\qquad\,\,\,\,{}+ t_{abc} I^{a(1)}_{-1} I^{b(2)}_{-1} I^{c(2)}_{-1} \vackk  \left(-\aaa 1 - 1\right)\left(-\aaa 1 -2\right) \left(-\aaa 2 - 2\right)  \\
&\quad\qquad\qquad{}- \ta t_{abc} I^{a(2)}_{-1} I^{b(2)}_{-1} I^{c(2)}_{-1} \vackk  \left(-\aaa 1\right) \left(-\aaa 1 - 1\right)\left(-\aaa 1 -2\right)\bigg),
\end{align*}
where the normalization factor $C(k_1,k_2)$ is given by
\begin{align*} C(k_1,k_2) &\coloneqq -\frac{M^3}{4}\frac{1}{(k_1+M)(k_2+M)(k_1+k_2+M)} D(k_1,k_2) ,\\
   D(k_1,k_2)^2 &\coloneqq \frac{-M}{ 2(M+2k_1)(M+2k_2)(3M+ 2k_1+2k_2)(M^2-4)}. \nn\end{align*}
\begin{proposition}\label{prop: TW}
Up to normalization factors depending only on the levels $k_1,k_2$, the states $\T$ and $\W$ are equal to $\int_\gamma \P(z)^{-1/\hc} \s_1(z) dz$ and $\int_\gamma \P(z)^{-2/\hc} \s_2(z) dz $ respectively. Namely,
\begin{align} \T &= \frac{ k_1k_2  } {(k_1+k_2+\hc)(k_1+k_2)(k_1+k_2-\hc)} 
\frac{\int_\gamma \P(z)^{-1/\hc} \s_1(z) dz }{{\int_\gamma \P(z)^{-1/\hc}dz }}, \nn\\
 \W &= \qquad\frac{C(k_1,k_2) (-\aaa1)_3 (-\aaa2)_3}{(-\aaa1-\aaa2 +1)_3}\qquad 
\frac{\int_\gamma \P(z)^{-2/\hc} \s_2(z) dz }{{\int_\gamma \P(z)^{-2/\hc}dz }}.\nn
\end{align}
\end{proposition} 
\begin{proof} From the definition, \S\ref{sec: S1S2def}, of $\s_1(z)$ we have, using Lemma \ref{lem: beta},
\begin{align}
\int_\gamma \P(z)^{-1/\hc} \s_1(z) dz = -\int_\gamma \P(z)^{-1/\hc} \left( \frac{1}{z} - \frac{1}{z-1} \right) dz \times \Xi
\end{align}
where one finds 
\be \Xi = - k_1 \T^{(2)} - k_2 \T^{(1)} + I^{a(1)}_{-1} I^{a(2)}_{-1} \vackk = - (k_1+k_2+\hc) \T. \nn\ee
The second equality is by a short calculation starting from the definition of $\T$ above. By further use of Lemma \ref{lem: beta}, one has the first part. 
Similarly, from the definition of $\s_2(z)$ in \S\ref{sec: S1S2def} we have 
\begin{align} \int_{\gamma} \P(z)^{-2/\hc} \s_2(z) dz 
&= 
\ta t_{abc} I^{a(1)}_{-1} I^{b(1)}_{-1} I^{c(1)}_{-1} \vackk B(-\aaa1 -3,-\aaa2) 
\nn\\&\,\,\,\,{}
+ t_{abc} I^{a(1)}_{-1} I^{b(1)}_{-1} I^{c(2)}_{-1} \vackk  B(-\aaa1-2,-\aaa2-1)
\nn \\&\,\,\,\,{}
+ t_{abc} I^{a(1)}_{-1} I^{b(2)}_{-1} I^{c(2)}_{-1} \vackk  B(-\aaa1 -1,-\aaa2-2)
\nn \\&{}
+ \ta t_{abc}  I^{a(2)}_{-1} I^{b(2)}_{-1} I^{c(2)}_{-1} \vackk B(-\aaa1,-\aaa2-3)\nn
\end{align}
and the second result follows by repeated use of Lemma \ref{lem: beta}.
\end{proof}

By the coset construction \cite{GKO}, the state $\T\in \VV_0^{\bm k}$ is a conformal vector: it generates a copy of the Virasoro vertex algebra with central charge
\be c = \dim(\slM) \left( \frac{k_1}{k_1+\hc} +\frac{k_2}{k_2+\hc} -\frac{k_1+k_2}{k_1+k_2+\hc} \right).\nn\ee  
That is,
\be \T_{(n)} \T = \delta_{n,0} T (\T) + \delta_{n,1} 2\T + \delta_{n,3} \ha c\vackk \nn\ee
for $n\in \ZZ_{\geq 0}$.

The state $\W$ was constructed in \cite{BBSS} (see equation 2.8 of that paper, where the field $Y[W, u]$ is given). Suppose we specialize further to the case of $M=3$, \emph{i.e.} $\widehat{\mathfrak{sl}}_3$, and choosing  the irreducible module at the marked point $z_2=1$ to have highest weight
\be \lambda_2 = \Lambda_0.\nn\ee 
so that, in particular, $k_2=1$.
It was shown in \cite{BBSS} that in that case the states $\W$ and $\T$ generate a copy of the $W_3$ algebra. That is, one has -- see \emph{e.g.} \cite{BMP} -- 
\be \T_{(n)} \W = \delta_{n,0} T (\W) + \delta_{n,1} 3\W,  \nn\ee
\begin{multline} \W_{(n)} \W = \delta_{n,0}\left( \beta T(\Lambda) + \frac 1{15} T^3(\T) \right)
  + \delta_{n,1} \left( 2 \beta \Lambda + \frac 3{10} T^2(\T) \right) \\
  + \delta_{n,2} T(\T) + \delta_{n,3} 2\T + \delta_{n,5} \ta c \vackk,
\end{multline}
for $n\in \ZZ_{\geq 0}$. Here
\be \Lambda \coloneqq \T_{(-1)}\T - \frac 3{10} T^2(\T),\quad\text{and}\quad
\beta = \frac{16}{22+5c} .\nn\ee

\begin{remark*}
The $W_3$ algebra is known to have a commutative algebra of Quantum Integral of Motions (of $\widehat{\sl}_3$-(m)-KdV, \emph{i.e.} quantum Boussinesq theory) indexed by the exponents of $\widehat{\sl}_3$. The first few integrals of motion $I_1,I_2,I_4,I_5$ can be found in \cite{BHK}.\footnote{Note that $\W_{BBSS} = \sqrt{3\beta} \W_{BHK}$.} The first two are $I_1 = \T_{\fm{0}}$ and $I_2= \W_{\fm{0}}$. By Proposition \ref{prop: TW}, these are proportional to our first two Hamiltonians $\hat Q_1^{\gamma}$ and $\hat Q_2^{\gamma}$. It is natural to think that the higher integrals of motion are likewise the two-point specializations of the (conjectural) higher Gaudin Hamiltonians, as was first conjectured in the $\widehat{\sl}_2$ case in \cite[\S6.4]{FFsolitons}. We have checked that the vacuum eigenvalues of $I_4$ and $I_5$ indeed agree with the eigenvalues $\int_\gamma \P(z)^{-4/3}v_4(z)dz$ and $\int_\gamma \P^{-5/3}(z)v_5(z) dz$ (see \S\ref{sec: op and eval}) predicted in \cite{LVY1}, up to overall factors depending on the remaining weight $\lambda_1$ only through the level $k_1$. The analogous check also works for $I_3$ and $I_5$ in the $\widehat{\sl}_2$ case.
\end{remark*}

\providecommand{\bysame}{\leavevmode\hbox to3em{\hrulefill}\thinspace}
\providecommand{\MR}{\relax\ifhmode\unskip\space\fi MR }
\providecommand{\MRhref}[2]{%
  \href{http://www.ams.org/mathscinet-getitem?mr=#1}{#2}
}
\providecommand{\href}[2]{#2}

\end{document}